\documentclass[a4paper,11pt, twoside]{article}
\usepackage{a4}
\usepackage{geometry} 
\usepackage{amsmath}
\usepackage{amssymb}
\usepackage{amsthm}
\usepackage{graphicx}
\usepackage{caption,subcaption}
\usepackage{multirow}
\usepackage{xstring}
\usepackage[utf8]{inputenc}
\usepackage{hyperref}
\usepackage{amstext,amsbsy,amsopn,eucal,enumerate}
\usepackage{tikz}
\usepackage{pgfplots}
\usepackage{tabularx}

\usepackage{soul}
\usepackage{algorithm} 
\usepackage{algpseudocode}

\def\R{\mathbb{R}}
\def\C{\mathbb{C}}

\newcommand{\expn}{\operatorname{e}}
\newcommand{\pro}{\operatorname{Pr}}
\newcommand{\diag}{\operatorname{diag}}

\newcommand{\orth}{\operatorname{orth}}
\newcommand{\vect}{\operatorname{vec}}

\newcommand{\beq}{\begin{equation}}
\newcommand{\eeq}{\end{equation}}
\newcommand {\mat}      [1] {\left[\begin{array}{#1}}
\newcommand {\rix}          {\end{array}\right]}
\newcommand {\smat}      [1] {\left[\begin{smallmatrix}{#1}}
\newcommand {\srix}          {\end{smallmatrix}\right]}
\newcommand {\s}      [1] {\begin{smallmatrix}{#1}}
\newcommand {\se}          {\end{smallmatrix}}
\newcommand{\trace}{\operatorname{tr}}

\newcommand{\hA}{\ensuremath{\hat{A}}}
\newcommand{\hB}{\ensuremath{\hat{B}}}
\newcommand{\hC}{\ensuremath{\hat{C}}}

\newcommand{\tB}{\ensuremath{\tilde{B}}}
\newcommand{\tC}{\ensuremath{\tilde{C}}}
\newcommand{\cE}{\ensuremath{\mathcal{E}}}

\newtheorem{defn}{Definition}[section]
\newtheorem*{remark}{Remark}

\newtheorem{lem}[defn]{Lemma}
\newtheorem{prop}[defn]{Proposition} 

\newtheorem{thm}[defn]{Theorem}

\usetikzlibrary{external}
\def\addlegendimage{\csname pgfplots@addlegendimage\endcsname}

\newlength\fheight
\newlength\fwidth

  \newcommand{\matlab}{MATLAB\textsuperscript{\textregistered}}
  \newcommand{\intel}{Intel\textsuperscript{\textregistered}}
    \newcommand{\xeon}{Xeon\textsuperscript{\textregistered}}

\title{Towards Time-Limited $\mathcal H_2$-Optimal Model Order Reduction}

\author{Pawan Goyal\thanks{Max Planck Institute for Dynamics of Complex Technical Systems, Sandtorstr. 1, 39106 Magdeburg, Germany, Email: {\tt 
			goyalp@mpi-magdeburg.mpg.de}.}\quad 
				Martin Redmann\thanks{Corresponding author. Weierstrass Institute for Applied Analysis and Stochastics, Mohrenstr. 39, 10117 Berlin, Germany, Email: {\tt 
martin.redmann@wias-berlin.de}.  Financial support by the DFG via Research Unit FOR 2402 is gratefully acknowledged.}}

\begin{document}

\maketitle

\begin{abstract}
In order to solve partial differential equations numerically and accurately,  a high order spatial discretization is usually needed. Model order 
reduction (MOR) techniques are often used to reduce the order of spatially-discretized systems and hence reduce computational 
complexity. A particular class of MOR techniques are $\mathcal H_2$-optimal methods such as the \emph{iterative rational Krylov subspace algorithm} 
(IRKA) and related schemes. However, these methods are used to obtain good approximations on a infinite time-horizon. Thus, in this work, our main 
goal is to discuss MOR schemes for time-limited linear systems. For this, we propose an alternative time-limited $\mathcal H_2$-norm
and show its connection with the time-limited Gramians.  We then provide first-order optimality conditions for an optimal reduced order model (ROM) 
with respect to the time-limited  $\mathcal H_2$-norm. Based on these optimality conditions, we propose an iterative scheme, which, upon convergence, 
aims at satisfying these conditions approximately. Then, we analyze how far away the obtained ROM due to the proposed algorithm is from satisfying the 
optimality conditions. We test the efficiency of the proposed iterative scheme using various numerical examples and illustrate that the newly proposed 
iterative method can lead to a better reduced-order compared to the unrestricted IRKA in the finite time interval of interest. 
\end{abstract}



\section{Introduction}

We consider a continuous linear time-invariant (LTI) system as follows:
\begin{equation}\label{sys:original}
\Sigma:\left\{ \begin{aligned}
\dot x(t) &= Ax(t) +  Bu(t),\quad x(0) = 0,\\
y(t) &= Cx(t),\quad t\geq 0,
\end{aligned}\right.
\end{equation}
where $A \in \R^{n\times n}$, $B \in \R^{n\times m}$, and $C \in \R^{p\times n}$. Generally, $x(t) \in \R^n$, $u(t)\in \R^m$ and $y(t) \in \R^p$ 
denote the state, control input and the quantity of interest (output vector) at time $t$, respectively, and in the most cases, the dimension of the 
state vector is much larger than the number of control inputs and outputs, i.e., $n \gg m,p$. We also assume that the matrix $A$ is Hurwitz, meaning 
that $\Lambda(A) \subset \C_{-}$, where $\Lambda(\cdot)$ denotes the spectrum of a matrix.  Due to the large dimension of system 
(\ref{sys:original}), 
it is numerically very expensive to simulate the system for various control inputs and perform engineering studies such as optimal control and 
optimization. One approach to 
overcome such an issue is \emph{model order reduction} (MOR), where we aim at constructing a reduced-order system as follows:
\begin{equation}\label{sys:reduced}
\hat\Sigma:\left\{ \begin{aligned}
\dot {\hat x}(t) &= \hat A \hat x(t) +  \hat Bu(t), \quad \hat x(0) = 0\\
\hat y(t) &= \hat C\hat x(t),\quad t\geq 0,
\end{aligned}\right.
\end{equation}
where  $\hat A\in \R^{r\times r}$, $\hat B \in \R^{r\times m}$, and $\hat C \in \R^{p\times r}$ and $r \ll n$ such that $y \approx \hat y$ in an 
appropriate norm for all admissible control inputs $u$. In the literature, there is a huge collection of methods available which allow us to 
construct such reduced-order systems, e.g., see \cite{morAnt05,morBenMS05,morSchVR08}. 

Most of the methods for linear systems such as balanced truncation, e.g., see \cite{morAnt05, morMoo81} and the iterative rational Krylov subspace 
algorithm \cite{morGugAB08} aim at constructing a reduced-order system which is good for an infinite time horizon. In other words, the output 
of system (\ref{sys:original}) is very well approximated by the output of (\ref{sys:reduced}) on the time interval $[0,\infty)$. 
However, there are several practical applications, as for example, a 
finite-time optimal control problem, where one is interested in approximating the output $y$ on a finite time interval, e.g., $[0,\bar T]$, meaning 
that \begin{equation}\label{eq:time_constrain}
y \approx \hat y\quad \text{on} \quad [0,\bar T].
\end{equation}
Due to relation (\ref{eq:time_constrain}), we expect a better reduced-order system in the time interval $[0,\bar T]$ as compared to unconstrained MOR 
approaches for a given order of the reduced system. Such a problem in a view of balanced truncation was 
first considered in \cite{morGawJ90} and its further studied was carried out in \cite{morK17, redmannkuerschner}. However, in this work, we consider 
a similar time-limited model reduction problem but rather in a view of extending the Wilson conditions \cite{wilson1970optimum} and 
 first-order optimality conditions \cite{morGugAB08,  meier1967approximation, wilson1970optimum}. \smallskip 

In Section \ref{sec:problemsetting}, we first propose the time-limited $\mathcal  H_2$-norm for linear systems 
and provide different representations of the metric induced by this norm which are based on time-limited Gramians. Then, 
we define the problem setting for time-limited MOR as an optimization problem. Subsequently, in Section \ref{sec:optimality_condtions}, we extend the 
Wilson conditions to time-limited linear systems and derive  first order optimality conditions, which 
minimize the time-limited $\mathcal H_2$-norm of the error system. Based on these conditions, we propose an iterative scheme, which, upon 
convergence, aims at constructing a reduced-order system, satisfying the optimality conditions approximately. Later on, we derive expressions, 
revealing how far away the obtained reduced systems via the proposed iterative scheme are from being locally optimal. In Section 
\ref{sec:numericalsection}, we illustrate the efficiency of the proposed iterative scheme by three benchmark numerical examples for linear systems. 
Finally, we conclude the paper with a short summary and an outlook for future work.

\section{Time-Limited $\boldsymbol{\mathcal H_2}$-Norm and Problem Setting} \label{sec:problemsetting}

In this section, we first define the time-limited $\mathcal H_2$-norm for linear systems and show its relation to the output error. Furthermore, we 
provide different representations for the time-limited $\mathcal H_2$-norm using time-limited Gramians and then define the time-limited 
$\mathcal H_2$-model reduction problem for linear systems. Before we proceed further, we note important relations between the Kronecker 
product, the vectorization and the trace of  a matrix. These are:
\begin{subequations}
\begin{align}\label{vec_kron_rel}
\vect(X Y Z) &= (Z^T\otimes X) \vect(Y) ,\\ \label{vec_trace_rel}
\trace(X Y Z) & = \left(\vect\left(X^T\right)\right)^T (I\otimes Y)\vect(Z),
\end{align}
\end{subequations}
 where $X, Y$ and $Z$ are matrices of suitable dimensions;  $\vect(\cdot)$ and $\trace(\cdot)$ denote the vectorization and the trace of a 
matrix, and $\otimes$ represents the Kronecker product of two matrices.
\smallskip

We investigate a model reduction problem for the large scale system \eqref{sys:original}; more precisely, we are seeking for a reduced-order system 
\eqref{sys:reduced} having the same structure. Since our goal is to construct a good approximation of the system \eqref{sys:original} on a finite 
time interval $[0, \bar T]$, where $\bar T>0$ is the terminal time, we first 
investigate the worst case error between the output of the system \eqref{sys:reduced} 
and the output of \eqref{sys:original} on $[0, \bar  T]$. In order to find a bound for the error between the output $y$ of the original model and the 
output $\hat y$ of the reduced system, arguments from the case of having an infinite time horizon are used, see, e.g., \cite{morAnt05, morGugAB08}. Similar 
estimates can be found in \cite{redmannbenner, redmannfreitag, BTtyp2EB}, where $\mathcal H_2$-error bounds for more general stochastic systems 
applying balanced truncation are derived.\smallskip

We make use of the explicit representations for the outputs \begin{align*}
 y(t)=C \int_0^t \expn^{A(t-s)} B u(s) ds, \quad \hat y(t)=\hat C \int_0^t  \expn^{\hat A(t-s)} \hat B u(s) ds,
\end{align*}
and obtain that {\allowdisplaybreaks \begin{align*}
\left\|y(t)- \hat y(t)\right\|_2 &=\left\|C \int_0^t \expn^{A(t-s)} B u(s) ds- \hat C \int_0^t  \expn^{\hat A(t-s)} \hat B u(s) 
ds\right\|_2\\&\leq \int_0^t \left\|\left(C \expn^{A(t-s)}  B - \hat C \expn^{\hat A(t-s)} \hat B\right) u(s)\right\|_2 ds\\&\leq \int_0^t \left\| 
C \expn^{A(t-s)} B - \hat C \expn^{\hat A(t-s)} \hat B\right\|_F \left\|u(s)\right\|_2 ds.
\end{align*}}
By the inequality of Cauchy-Schwarz and substitution, we have \begin{align*}
\left\|y(t)- \hat y(t)\right\|_2 &\leq \left(\int_0^t \left\|C \expn^{A(t-s)} B - \hat C \expn^{\hat A(t-s)} \hat B\right\|_F^2 
ds\right)^{\frac{1}{2}} \left(\int_0^t \left\|u(s)\right\|_2^2 ds\right)^{\frac{1}{2}}\\&\leq \left(\int_0^t \left\|C \expn^{As} B - \hat C 
\expn^{\hat As} \hat B\right\|_F^2 ds\right)^{\frac{1}{2}}  \left(\int_0^t \left\|u(s)\right\|_2^2 ds\right)^{\frac{1}{2}} \\&\leq \left(\int_0^{\bar 
T} \left\|C \expn^{As} B - \hat C \expn^{\hat A s} \hat B\right\|_F^2 ds\right)^{\frac{1}{2}}  \left\|u\right\|_{L^2_{\bar T}}
\end{align*}
for $t\in[0, \bar T]$. Hence, \begin{align}\label{relation_H2_output}
\max_{t\in [0, \bar T]}\left\|y(t)- \hat y(t)\right\|_2 \leq \left\|\Sigma-\hat\Sigma\right\|_{\mathcal H_{2, \bar T}}  \left\|u\right\|_{L_{\bar 
T}^2},
\end{align}
where $\left\|\Sigma-\hat\Sigma\right\|_{\mathcal H_{2, \bar T}}:=\left(\int_0^{\bar T} \left\|C \expn^{As} B - \hat C \expn^{\hat A s} \hat 
B\right\|_F^2 ds\right)^{\frac{1}{2}}$. We call $\left\|\cdot\right\|_{\mathcal H_{2, \bar T}}$ the time-limited $\mathcal H_2$-norm since  
$\left\|\Sigma-\hat\Sigma\right\|_{\mathcal H_{2, \bar T}}$ provides the time-domain representation of the metric induced by the $\mathcal H_2$-norm 
if $\bar T\rightarrow \infty$.\smallskip

The time-limited $\mathcal H_2$-error can also be expressed with the help of the time-limited reachability and observability Gramians. We refer, 
e.g., to \cite{morGawJ90} for a further discussion of these Gramians. In order to show the Gramian based representations, we first provide the 
following lemma.  

\begin{lem}\label{lem:time_lim_sylvester}
Let $A_1\in\mathbb{R}^{d_1\times d_1},~A_2\in\mathbb{R}^{d_2\times d_2}$ with $\Lambda(A_1)\cap\Lambda(-A_2)=\emptyset$ and
$K_1\in\mathbb{R}^{d_1\times d_3}$, $K_2\in\mathbb{R}^{d_2\times d_3}$. Then, \begin{align*}
 X=\int_{0}^{\bar T}\expn^{A_1s}K_1K_2^T \expn^{A_2^Ts}ds
\end{align*}
uniquely solves the Sylvester equation
\begin{align}\label{eq:sylvester}
 A_{1} X+X A_{2}^T &=-K_1 K_2^T+\expn^{A_1{\bar T}}K_1K_2^T\expn^{A_2^T{\bar T}}.
\end{align}
\begin{proof}
This result is a consequence of the product rule. Setting $g_1(t):=\expn^{A_1 t}K_1$ and $g_2(t):=K_2^T \expn^{A_2^T t}$, it 
holds that \begin{align*}
&g_1(\bar T)g_2(\bar T)-g_1(0)g_2(0)=\int_{0}^{\bar T}d g_1(s) g_2(s) + \int_{0}^{\bar T}g_1(s)dg_2(s)\\
&= A_1 \int_{0}^{\bar T} g_1(s) g_2(s) ds + \int_{0}^{\bar T}g_1(s)g_2(s)ds\;A_2^T = A_1X +  XA_2^T,      
 \end{align*}
since $dg_2(s)=g_2(s) A_2^T ds$ and $dg_1(s)=A_1 g_1(s) ds$. Furthermore, using~\eqref{vec_kron_rel},  equation~\eqref{eq:sylvester} can be 
written equivalently as 
\begin{align}\label{eq:vec_kron}
\underbrace{\left(I_{d_2}\otimes A_1+A_2\otimes I_{d_1}\right) }_{=:\mathcal A_\otimes}\vect(X)=\vect(R_{12}),
\end{align}
where $R_{12}$ is the right-hand side in (\ref{eq:sylvester}) and $I_q$ denotes the identity matrix of size 
$q\times q$. Now, the eigenvalues of $\mathcal A_\otimes$ are given by $\mu_1^{(i)}+\mu_2^{(j)}$, where $\mu_1^{(i)}$ is the $i$th eigenvalue of $A_1$ 
and $\mu_2^{(j)}$ the $j$th eigenvalue of $A_2$. Due the assumption on the spectra of $A_1$ and $A_2$, the matrix $\mathcal A_\otimes$ is invertible 
which gives a unique solution to (\ref{eq:vec_kron}).
\end{proof}
       \end{lem}
The next proposition shows that the time-limited error can be expressed with the help of time-limited Gramians. This result is used later on in order 
to derive first-order necessary conditions for a minimal error in the time-limited $\mathcal H_2$-norm.
\begin{prop}\label{prop:reach_rep}
	Let $\Sigma$ and $\hat \Sigma$ be the original and reduced-order systems as defined in \eqref{sys:original} and \eqref{sys:reduced}. Then, 
the time-limited $\mathcal H_2$-norm of $\Sigma-\hat{\Sigma}$ is given by 
\begin{align}\label{eq:errorequation}
\left\|\Sigma-\hat\Sigma\right\|^2_{\mathcal H_{2, \bar T}}=\trace(C P_{\bar T} C^T)+  \trace(\hat C \hat P_{\bar T} \hat C^T) - 2 
\trace(C P_{2, \bar T} \hat C^T),                                          
\end{align}
where  $P_{\bar T}, P_{2, \bar T}$ and $\hat P_{\bar T}$, respectively, satisfy 
\begin{align}\label{full_reach}
A P_{\bar T}+ P_{\bar T} A^T &=-B B^T+\expn^{A \bar T}B B^T \expn^{A^T \bar T},\\ \label{cross_reach}
A P_{2, {\bar T}}+ P_{2, {\bar T}} \hat A^T &=-B \hat B^T+\expn^{A \bar T}B \hat B^T \expn^{\hat A^T \bar T}, \\ \label{reduced_reach}
\hat A \hat P_{\bar T}+\hat P_{\bar T} \hat A^T &=-\hat B \hat B^T+\expn^{\hat A \bar T}\hat B \hat B^T \expn^{\hat A^T \bar T}.                     
\end{align}
\end{prop}
 \begin{proof}
 The definition of the Frobenius norm and the linearity of the integral yield {\allowdisplaybreaks \begin{align*}
 &\left\|\Sigma-\hat\Sigma\right\|^2_{\mathcal H_{2, \bar T}}= \int_0^{\bar T} \left\|C \expn^{As} B - \hat C \expn^{\hat As} \hat B\right\|_F^2 
ds\\&
=\int_0^{\bar T} \trace\left(C \expn^{As}BB^T \expn^{A^Ts} C^T\right)ds+\int_0^{\bar T}\trace\left(\hat C \expn^{\hat As}\hat 
B\hat B^T\expn^{\hat A^Ts}  \hat C^T\right)ds\\&\ \ \ \ -2\int_0^{\bar T}\trace\left(C \expn^{As}B\hat B^T \expn^{\hat A^Ts} \hat C^T\right) 
ds\\&=\trace\left(C P_{\bar T} C^T\right)+\trace\left(\hat C \hat P_{\bar T} \hat C^T\right)-2\;\trace\left(C P_{2, {\bar T}} \hat C^T\right),
\end{align*}}
with $P_{\bar T}:=\int_0^{\bar T}  \expn^{As}BB^T \expn^{A^Ts}ds$, $P_{2, \bar T}:=\int_0^{\bar T}  \expn^{As}B\hat B^T \expn^{\hat A^Ts}ds$, $\hat 
P_{\bar T}:=\int_0^{\bar T} \expn^{\hat As}\hat B\hat B^T\expn^{\hat A^Ts} ds$.
Due to Lemma \ref{lem:time_lim_sylvester} $P_{\bar T}, P_{2, \bar T}$ and $\hat P_{\bar T}$ are the solutions to (\ref{full_reach}), 
(\ref{cross_reach}) and (\ref{reduced_reach}), respectively.
 \end{proof}
The result of Proposition \ref{prop:reach_rep} has the same structure as the error in \cite{redmannkuerschner}, where the case of time-limited 
balanced truncation has been investigated. Moreover, if we take the limit $\bar T\rightarrow \infty$ in (\ref{eq:errorequation}), we obtain a 
representation for the full $\mathcal{H}_2$-error that is, e.g., derived in \cite{morAnt05}. The next proposition shows that the time-limited 
$\mathcal H_2$-norm of the error system as in Proposition \ref{prop:reach_rep} can be rewritten using the time-limited observability Gramians.
\begin{prop}\label{prop:express_error_obs}
Let $\Sigma$ and $\hat \Sigma$ be the original and reduced-order systems as defined in \eqref{sys:original} and \eqref{sys:reduced}. Moreover, let 
$P_{\bar T}, P_{2, \bar T}$ and $\hat P_{\bar T}$ be the solutions to \eqref{full_reach},  \eqref{cross_reach} and \eqref{reduced_reach}, 
respectively. Then, the following holds: 
 \begin{align*}
\trace(C P_{\bar T} C^T)&=\trace(B^T Q_{\bar T} B),\\  \trace(\hat C \hat P_{\bar T} \hat C^T)&=\trace(\hat B^T \hat Q_{\bar T} \hat B),\\ 
\trace(C P_{2, \bar T} \hat C^T)&=\trace(\hat B^T Q_{2, \bar T} B),                                          
                                                             \end{align*}
 where the matrices $Q_{\bar T}, Q_{2, \bar T}$ and $\hat Q_{\bar T}$ satisfy \begin{align}\label{full_obs}
A^T Q_{\bar T}+ Q_{\bar T} A &=-C^T C+\expn^{A^T \bar T}C^T C \expn^{A \bar T},\\ \label{cross_obs}
\hat A^T Q_{2, {\bar T}}+ Q_{2, {\bar T}} A &=-\hat C^T C+\expn^{\hat A^T \bar T} \hat C^T C \expn^{A \bar T}, \\ \label{reduced_obs}
\hat A^T \hat Q_{\bar T}+\hat Q_{\bar T} \hat A &=-\hat C^T \hat C+\expn^{\hat A^T \bar T}\hat C^T \hat C \expn^{\hat A \bar T}.
                                          \end{align}
                                          \begin{proof}
We insert the integral representations of $P_{\bar T}, P_{2, \bar T}$ and $\hat P_{\bar T}$ and use basic properties of the trace operator. Thus, 
\begin{align*}
\trace(C P_{\bar T} C^T)&=\int_0^{\bar T}  \trace(C \expn^{As}BB^T \expn^{A^Ts}C^T)ds=\int_0^{\bar T}  \trace(B^T \expn^{A^Ts}C^TC \expn^{As}B)ds,\\  
\trace(\hat C \hat P_{\bar T} \hat C^T)&=\int_0^{\bar T} \trace(\hat C \expn^{\hat As}\hat B\hat B^T\expn^{\hat A^Ts} \hat C^T) ds=\int_0^{\bar T} 
\trace(\hat B^T\expn^{\hat A^Ts} \hat C^T \hat C \expn^{\hat As}\hat B) ds,\\ 
\trace(C P_{2, \bar T} \hat C^T)&=\int_0^{\bar T}  \trace(C\expn^{As}B\hat B^T \expn^{\hat A^Ts} \hat C^T)ds=\int_0^{\bar T}  \trace(\hat 
B^T\expn^{\hat A^Ts}\hat C^T C\expn^{As}B)ds.                                          
                                                             \end{align*}
Let us define $Q_{\bar T}:=\int_0^{\bar T}  \expn^{A^Ts}C^TC \expn^{As}ds$, $Q_{2, \bar T}:=\int_0^{\bar T}  \expn^{\hat A^Ts}\hat C^TC 
\expn^{As}ds$ and $\hat Q_{\bar T}:=\int_0^{\bar T} \expn^{\hat A^Ts} \hat C^T \hat C \expn^{\hat As} ds$. Then, applying Lemma 
\ref{lem:time_lim_sylvester} yields the claim.
\end{proof}
\end{prop}
From inequality~\eqref{relation_H2_output}, it can be seen that it makes sense to minimize $\left\|\Sigma-\hat\Sigma\right\|^2_{\mathcal H_{2, \bar 
T}}$ with respect to the reduced order matrices $\hat A$, $\hat B$ and $\hat C$ since a small $\mathcal H_{2, \bar T}$-error ensures a small output 
error. Due to the fact that $\left\|\Sigma-\hat\Sigma\right\|_{\mathcal H_{2, \bar T}}$ is increasing in $\bar T$, the time-limited 
error is less or equal to the error in the full $\mathcal H_2$-norm $\left\|\cdot\right\|_{\mathcal H_{2, \infty}}$. Thus, 
$\left\|\cdot\right\|_{\mathcal H_{2, \bar T}}$ provides a more accurate bound than $\left\|\cdot\right\|_{\mathcal H_{2, \infty}}$ for the output 
error in~\eqref{relation_H2_output}. By minimizing $\left\|\cdot\right\|_{\mathcal H_{2, \bar T}}$, we hope to find a reduce order 
model on $[0, \bar T]$ with an accuracy that is better than in the case of having a locally optimal reduced system with respect to 
$\left\|\cdot\right\|_{\mathcal H_{2, \infty}}$. 

\section{First-Order Necessary Conditions for Optimality and Model Order Reduction}\label{sec:optimality_condtions}

In this section, we begin by deriving first-order necessary conditions for time-limited $\mathcal H_2$-optimal reduced order systems. In other 
words, our aim is to construct a reduced-order system $\hat{\Sigma}$ of order $r$ as in \eqref{sys:reduced}, such that it minimizes 
$\|\Sigma-\hat{\Sigma}\|^2_{\mathcal H_{2,\bar T}} =: \mathcal E$, where $\Sigma$ is the original system as in \eqref{sys:original}. An expression 
for $\mathcal E$ is given in \eqref{eq:errorequation}.  Since the term $\trace(C P_{\bar T} C^T)$ in \eqref{eq:errorequation} does not depend on the 
reduced order matrices, we focus on minimizing the expression\begin{align}\label{reduce_expression}  
\mathcal E_r := \trace(\hat C \hat P_{\bar T} \hat C^T) - 2 \trace(C P_{2, \bar T} \hat C^T).
\end{align} 
Before proceeding further, we assume that the matrix $\hat A$ in (\ref{sys:reduced}) is diagonalizable, i.e., there exists an invertible matrix $S$ 
such that $\hat A=S^{-1} D S$, where $D=\diag(\lambda_1, \ldots, \lambda_r)$. Using the 
matrix $S$ as a state-space transformation of \eqref{sys:reduced}, the term (\ref{reduce_expression}) can be equivalently rewritten as 
\begin{align}\nonumber 
\mathcal E_r&=\trace(\hat C S^{-1} S \hat P_{\bar T}S^TS^{-T} \hat C^T) - 2 \trace(C 
P_{2, \bar T} S^T S^{-T}\hat C^T)\\
&=\trace(\tilde C \tilde P_{\bar T} \tilde C^T) - 2 \trace(C \tilde P_{2, \bar T}  \tilde C^T),\label{trans_min}
\end{align}
where $\tilde C=\hat C S^{-1}$, $\tilde P_{\bar T}=S \hat P_{\bar T}S^T$ and $\tilde P_{2, \bar T}=P_{2, \bar T} S^T$. Furthermore, it can be shown
that the matrices $\tilde P_{\bar T}$ and  $\tilde P_{2, \bar T}$ are the solutions to
\begin{align}\label{trans_full_reach}
A \tilde P_{2, {\bar T}}+ \tilde P_{2, {\bar T}} D &=-B \tilde B^T+\expn^{A \bar T}B \tilde B^T \expn^{D \bar T}, \\ 
\label{trans_reduced_reach}
D\tilde P_{\bar T}+\tilde P_{\bar T} D &=-\tilde B \tilde B^T+\expn^{D \bar T}\tilde B \tilde B^T \expn^{D \bar T},                   
\end{align}
respectively, where $\tilde B=S\hat B$. Precisely, Equation (\ref{trans_full_reach}) is obtained by multiplying (\ref{cross_reach})
with $S^T$ from the right side, and Equation  \eqref{trans_reduced_reach} is derived by multiplying (\ref{reduced_reach}) with
$S$ and $S^T$ from the left and the right side, respectively, and using the relation $\expn^{\hat A \bar T}=S^{-1} \expn^{D \bar T}S$. \smallskip

In order to find necessary conditions for a locally minimal transformed error expression (\ref{trans_min}), we compute the partial derivatives of the 
form $\partial_x \trace(\tilde C \tilde P_{\bar T} \tilde C^T)$ and $\partial_x \trace(C \tilde P_{2, \bar T} \tilde C^T)$ and then set \begin{align*}
 \partial_x \trace(\tilde C \tilde P_{\bar T} \tilde C^T)=  2 \partial_x \trace(C \tilde P_{2, \bar T} \tilde C^T),                                    
                                                                         \end{align*}
where $x=\lambda_i, \tilde c_{ki}, \tilde b_{ij}$, $i \in \{1,\ldots,r\}$, $j \in \{1,\ldots,m\}$, $k \in \{1,\ldots,p\}$ and $\tilde c_{ki}$, 
$\tilde b_{ij}$ being $kj$-th and $ij$-th elements of the matrices $\tilde C$ and $\tilde B$, respectively.

Let us start with the optimality conditions with respect to $\tilde c_{ki}$. With $e_i$, we denote the $i$-th column of the identity matrix of 
suitable dimension that is clear from the context. We then obtain that \begin{align*}
        \partial_{\tilde c_{ki}} \trace(\tilde C \tilde P_{\bar T} \tilde C^T) &=   \partial_{\tilde c_{ki}} \trace(\tilde C^T\tilde C \tilde P_{\bar T})\\
        &= \trace((\partial_{\tilde c_{ki}}\tilde C^T)\tilde C \tilde P_{\bar 
T}+\tilde C^T (\partial_{\tilde c_{ki}}\tilde C) \tilde P_{\bar T})=\trace(e_ie_k^T\tilde C \tilde P_{\bar T}+\tilde C^T e_k e_i^T \tilde 
P_{\bar T})\\
&=2 e_k^T\tilde C \tilde P_{\bar T}e_i,
                                                                                                                 \end{align*}
where we have used the linearity of the trace, the product rule and the fact that $\tilde P_{\bar T}$ does not depend on $\tilde C$. Since 
\begin{align*}
\partial_{\tilde c_{ki}} \trace(C \tilde P_{2, \bar T} \tilde C^T)=\trace(C \tilde P_{2, \bar T} e_ie_k^T)= e_k^T C \tilde P_{2, \bar T} e_i,
\end{align*}
the optimality condition with respect to $\tilde c_{ki}$ is $e_k^T \tilde C \tilde P_{\bar T} e_i=e_k^T C \tilde P_{2, \bar T} e_i$ for all $i \in 
\{1,\ldots,r\}$,  $k \in \{1,\ldots,p\}$. 
Hence, we obtain \begin{align}\label{opt1}
  \tilde C \tilde P_{\bar T} = C \tilde P_{2, \bar T}.               
                 \end{align}
We now derive the partial derivatives with respect to $\tilde b_{ij}$. We rewrite (\ref{trans_min}) to simplify this procedure by applying Proposition 
\ref{prop:express_error_obs}: \begin{align*}
      \mathcal E_r = \trace(\tilde C \tilde P_{\bar T} \tilde C^T) - 2 \trace(C \tilde P_{2, \bar T}  \tilde C^T)&= \trace(\hat B^T \hat Q_{\bar T} \hat B)-2 
\trace(\hat B^T Q_{2, \bar T} B)\\&= \trace(\tilde B^T \tilde Q_{\bar T} \tilde B)-2 \trace(\tilde B^T \tilde Q_{2, \bar T} B),
                              \end{align*}
where $\tilde Q_{\bar T}=S^{-T}\hat Q_{\bar T}S^{-1}$ and $\tilde Q_{2, \bar T}=S^{-T}\hat Q_{2, \bar T}$. The matrices $\tilde Q_{\bar T}$ and 
$\tilde Q_{2, \bar T}$ satisfy \begin{align}\label{trans_cross_obs}
D \tilde Q_{2, {\bar T}}+ \tilde Q_{2, {\bar T}} A &=-\tilde C^T C+\expn^{D \bar T} \tilde C^T C \expn^{D \bar T}, \\ 
\label{trans_reduced_obs}
D \tilde Q_{\bar T}+\tilde Q_{\bar T} D &=-\tilde C^T \tilde C+\expn^{D \bar T}\tilde C^T \tilde C \expn^{D \bar T},
                                          \end{align}
respectively. Again, Equation (\ref{trans_cross_obs}) is obtained by multiplying (\ref{cross_obs}) with $S^{-T}$ from the left side, and we find 
\eqref{trans_reduced_obs} by multiplying (\ref{reduced_obs}) with $S^{-T}$ from the left side and with $S^{-1}$ from the right side. Thus, we have 
\begin{align*}
  \partial_{\tilde b_{ij}}  \trace(\tilde B\tilde B^T \tilde Q_{\bar T})&= \trace((\partial_{\tilde b_{ij}} \tilde B)\tilde B^T \tilde Q_{\bar 
T}+ \tilde B(\partial_{\tilde b_{ij}}\tilde B^T) \tilde Q_{\bar T})= \trace(e_i e_j^T\tilde B^T \tilde Q_{\bar 
T}+ \tilde B e_j e_i^T \tilde Q_{\bar T})\\&= 2 e_i^T \tilde Q_{\bar T}\tilde B e_j 
    \end{align*}
    using that $\tilde Q_{\bar T}$ does not depend on $\tilde B$ or $\tilde b_{ij}$. Since \begin{align*}
\partial_{\tilde b_{ij}} \trace(\tilde B^T \tilde Q_{2, \bar T} B)= \trace(e_j e_i^T \tilde Q_{2, \bar T} B)= e_i^T \tilde Q_{2, \bar T} B e_j,
\end{align*}
it is necessary that $e_i^T \tilde Q_{\bar T}\tilde B e_j = e_i^T \tilde Q_{2, \bar T} B e_j$ for $i \in \{1,\ldots,r\}$, $j \in \{1,\ldots,m\}$, 
which can be equivalently written as
\begin{align}\label{opt2}
 \tilde Q_{\bar T}\tilde B  =  \tilde Q_{2, \bar T} B.
\end{align}
Next, we first introduce the following lemma in order to derive an optimality condition with respect to the eigenvalues $\lambda_i$ of $\hat A$.
\begin{lem}\label{lem:eq:derivatives} 
The partial derivatives $X^{(i)}:=\partial_{\lambda_i} \tilde P_{\bar T}$ and $X_2^{(i)}:=\partial_{\lambda_i} \tilde P_{2, \bar T}$ solve 
\begin{align}\label{bgpk}
&DX^{(i)}+X^{(i)} D =-e_ie_i^T\tilde P_{\bar T}-\tilde P_{\bar T}e_ie_i^T+\bar T 
e_ie_i^T\expn^{D \bar T}\tilde B \tilde B^T \expn^{D \bar T}+\bar T\expn^{D \bar T}\tilde B \tilde B^T \expn^{D \bar T}e_ie_i^T,                 \\
&A X_2^{(i)}+ X_2^{(i)} D =-\tilde P_{2, {\bar T}} e_ie_i^T+\bar T \expn^{A 
	\bar T}B \tilde B^T \expn^{D \bar T}e_ie_i^T, \label{bgpk2}
                                       \end{align}
respectively.
\begin{proof}
The derivative of the left side of equation (\ref{trans_full_reach}) is\begin{align*}
A X_2^{(i)}+ X_2^{(i)} D + \tilde P_{2, {\bar T}}e_i e_i^T
\end{align*} applying the product rule. The derivative of the corresponding right side is \begin{align*}
\expn^{A \bar T}B \tilde B^T \partial_{\lambda_i} \expn^{D \bar T} = \expn^{A \bar T}B \tilde B^T \expn^{D \bar T} e_i e_i^T 
\bar T, 
\end{align*}
because $\partial_{\lambda_i} \expn^{D \bar T}=\partial_{\lambda_i} \diag(\expn^{\lambda_1 \bar T}, \ldots, \expn^{\lambda_i \bar T}, \ldots, 
\expn^{\lambda_r \bar T})=\diag(0, \ldots, \bar T\expn^{\lambda_i \bar T}, \ldots, 0)$. This yields (\ref{bgpk}). Applying $\partial_{\lambda_i}$ to 
the left of equation (\ref{trans_reduced_reach}) provides
\begin{align*}
e_ie_i^T\tilde P_{\bar T}+D X^{(i)}+X^{(i)} D+\tilde P_{\bar T} e_ie_i^T                   
                                       \end{align*}
again using the product rule. Doing the same with the corresponding right side, we have \begin{align*}
  \partial_{\lambda_i}(\expn^{D \bar T}\tilde B \tilde B^T \expn^{D \bar T})&=    (\partial_{\lambda_i}\expn^{D \bar T})\tilde B \tilde B^T \expn^{D 
\bar T}+ \expn^{D \bar T}\tilde B \tilde B^T (\partial_{\lambda_i}\expn^{D \bar T})\\ &=   \bar T e_i e_i^T\expn^{D \bar T}\tilde B \tilde B^T 
\expn^{D \bar T}+ \expn^{D \bar T}\tilde B \tilde B^T \expn^{D \bar T} e_i e_i^T \bar T.                                                           
                                 \end{align*}
This provides (\ref{bgpk2}).
\end{proof}
\end{lem}
Before we proceed further, let us introduce the infinite Gramian $\tilde Q_{\infty}$, which we define as the solution to \begin{align}\label{infobsgram}  
D \tilde Q_{\infty}+\tilde Q_{\infty} D &=-\tilde C^T \tilde C.
                                          \end{align}    
It is well-defined if $D$ and $-D$ have no common eigenvalues. We insert matrix equation (\ref{infobsgram}) to
\begin{align*}
\partial_{\lambda_i}\trace(\tilde C \tilde P_{\bar T} \tilde C^T)=\trace(\tilde C^T\tilde C X^{(i)})=-\trace([D \tilde Q_{\infty}+\tilde Q_{\infty}D] 
X^{(i)})=-\trace(\tilde Q_{\infty}[X^{(i)} D +D X^{(i)}]).
\end{align*}
With Lemma \ref{lem:eq:derivatives}, we get \begin{align*}
\partial_{\lambda_i}\trace(\tilde C \tilde P_{\bar T} \tilde C^T)&= \trace(\tilde Q_{\infty}[e_ie_i^T\tilde P_{\bar T}+\tilde P_{\bar T}e_ie_i^T-\bar 
T e_ie_i^T\expn^{D \bar T}\tilde B \tilde B^T \expn^{D \bar T}-\bar T\expn^{D \bar T}\tilde B \tilde B^T \expn^{D \bar T}e_ie_i^T])\\&= 
2 e_i^T\tilde Q_{\infty}[\tilde P_{\bar T}-\bar T\expn^{D \bar T}\tilde B \tilde B^T \expn^{D \bar T}]e_i.
\end{align*}
Assuming that $D$ and $-A$ have no common eigenvalues, we define the infinite cross Gramian $\tilde Q_{2, \infty}$ which satisfies 
\begin{align*}
D \tilde Q_{2, {\infty}}+ \tilde Q_{2, {\infty}} A^T =-\tilde C^T C.
\end{align*}
Hence, it holds that \begin{align*}
\partial_{\lambda_i}\trace(C \tilde P_{2 \bar T} \tilde C^T)&=\trace(\tilde C^TC X_2^{(i)})=-\trace([D \tilde Q_{2, \infty}+\tilde Q_{2, \infty}A] 
X_2^{(i)})\\&=-\trace(\tilde Q_{2, \infty}[X_2^{(i)} D +A X_2^{(i)}])=\trace(\tilde Q_{2, \infty}[\tilde P_{2, {\bar T}}-\bar T \expn^{A 
\bar T}B \tilde B^T \expn^{D \bar T}]e_ie_i^T)\\&=e_i^T\tilde Q_{2, \infty}[\tilde P_{2, {\bar T}}-\bar T \expn^{A 
\bar T}B \tilde B^T \expn^{D \bar T}]e_i
\end{align*}
applying Lemma \ref{lem:eq:derivatives} again. This leads to the third optimality condition which is \begin{align}\label{opt3}
e_i^T\tilde Q_{2, \infty}[\tilde P_{2, {\bar T}}-\bar T \expn^{A 
\bar T}B \tilde B^T \expn^{D \bar T}]e_i=
 e_i^T\tilde Q_{\infty}[\tilde P_{\bar T}-\bar T\expn^{D \bar T}\tilde B \tilde B^T \expn^{D \bar T}]e_i
\end{align}
for all $i\in\{1, \ldots, r\}$.\smallskip

Below, the generalized optimality conditions are summarized that have been derived above. Additionally, we provide an 
equivalent Kronecker formulation in the next theorem that is useful for the error analysis in the optimality conditions.

A different  type of extended Wilson conditions for bilinear systems has been shown in \cite{morZhaL02}. Its equivalent Kronecker formulation is presented in 
\cite{morBenB12b}. Since the bilinear setting is very different from the time-limited case, the optimality conditions have a different structure  which can be seen in the next theorem. 
\begin{thm}\label{thm:opt_cond}
 Let the reduced-order system (\ref{sys:reduced}) be a locally optimal approximation to the original system (\ref{sys:original}) with respect to 
$\left\|\cdot\right\|_{\mathcal H_{2, \bar T}}$. Then, conditions (\ref{opt1}), (\ref{opt2}) and (\ref{opt3}) hold or equivalently, we have 
\begin{equation}\begin{aligned}\label{kron_opt1}
&(I\otimes \hat C)  \left[(I\otimes \hat A)+(D\otimes I)\right]^{-1}(\expn^{D \bar T}\tilde B \otimes \expn^{\hat 
A \bar T}\hat B -\tilde B \otimes \hat B) \vect(I)\\ &=
 (I\otimes C)  \left[(I\otimes A)+(D\otimes I)\right]^{-1} (\expn^{D \bar T}\tilde B 
\otimes \expn^{ A \bar T} B -\tilde B \otimes B) \vect(I),\end{aligned}\end{equation}
        \begin{equation}  \begin{aligned}\label{kron_opt2}
&(\hat B^T\otimes I) \left[(I\otimes D)+(\hat A^T\otimes 
I)\right]^{-1}(\expn^{\hat A^T \bar T}\hat C^T \otimes \expn^{D \bar T}\tilde C^T-\hat C^T \otimes \tilde C^T)\vect(I)\\ 
&= (B^T\otimes I) \left[(I\otimes D)+(A^T\otimes I)\right]^{-1}(\expn^{A^T \bar T} C^T \otimes \expn^{D 
\bar T}\tilde C^T-C^T \otimes \tilde C^T)\vect(I)\end{aligned}\end{equation}
and for all $i=1, \ldots, r$
    \begin{equation}\label{kron_opt3}\begin{aligned}  
  &\vect^T(I) (\hat C \otimes \tilde C) \left[(I\otimes D)+(\hat 
A\otimes I)\right]^{-1} (I\otimes e_i e_i^T) \\  & \quad \times \left( \left[(I\otimes 
D)+(\hat A\otimes I)\right]^{-1}  (\expn^{\hat A \bar T} 
\hat B \otimes \expn^{D \bar T}\tilde B -\hat B \otimes \tilde B)- (\bar T\expn^{\hat A \bar T}\hat B \otimes \expn^{D \bar T}\tilde B)  \right) 
\vect(I) 
\\&=  \vect^T(I) ( C \otimes \tilde C) \left[(I\otimes D)+(A\otimes I)\right]^{-1} (I\otimes e_i e_i^T) \\ & \quad 
\times \left( \left[(I\otimes D)+(A\otimes I)\right]^{-1}  (\expn^{A \bar T} B \otimes \expn^{D \bar T}\tilde B -B \otimes \tilde 
B)-(\bar T\expn^{A \bar T}B \otimes \expn^{D \bar T}\tilde B)\right) \vect(I).
\end{aligned}
\end{equation}
 \begin{proof}
 Applying the $\vect$ operator to (\ref{opt1}) leads to the following equivalent formulation:
 \begin{align*}
\vect(\tilde C \tilde P_{\bar T}) = \vect(C \tilde P_{2, \bar T}).
\end{align*}
Now, using the vectorization of  (\ref{trans_reduced_reach}) and the relation in (\ref{vec_kron_rel}), we obtain 
   \begin{align*}
 & \vect(\tilde C \tilde P_{\bar T})=(I\otimes \tilde C) \vect(\tilde P_{\bar T})= (I\otimes \tilde C)  \left[(I\otimes D)+(D\otimes 
I)\right]^{-1}\vect(\expn^{D \bar T}\tilde B \tilde B^T \expn^{D \bar T}-\tilde B \tilde B^T)\\&= 
(I\otimes \tilde C)  \left[(I\otimes D)+(D\otimes I)\right]^{-1}(\expn^{D \bar T}\tilde B \otimes \expn^{D \bar T}\tilde B -\tilde B \otimes 
\tilde B) \vect(I).    
                 \end{align*}
Since $(I\otimes \tilde C)=(I\otimes \hat C)(I\otimes S)^{-1}$ and $(\expn^{D \bar T}\tilde B \otimes \expn^{D \bar T}\tilde B -\tilde B \otimes 
\tilde B)=(I\otimes S^{-1})^{-1} (\expn^{D \bar T}\tilde B \otimes \expn^{\hat A \bar T}\hat B -\tilde B \otimes \hat B)$, we get \begin{align*}
 & \vect(\tilde C \tilde P_{\bar T})=(I\otimes \hat C)  \left[(I\otimes \hat A)+(D\otimes I)\right]^{-1}(\expn^{D \bar T}\tilde B \otimes \expn^{\hat 
A \bar T}\hat B -\tilde B \otimes \hat B) \vect(I).    
                 \end{align*}
With the help of (\ref{trans_full_reach}), the vectorization of $C \tilde P_{2, \bar T}$ is given by 
\begin{align*}
 &\vect(C \tilde P_{2, \bar T})= (I\otimes C) \vect(\tilde P_{2, \bar T})= (I\otimes C)  \left[(I\otimes A)+(D\otimes 
I)\right]^{-1}\vect(\expn^{A \bar T} B \tilde B^T \expn^{D \bar T}-B \tilde B^T)\\&= (I\otimes C)  \left[(I\otimes A)+(D\otimes 
I)\right]^{-1} (\expn^{D \bar T}\tilde B \otimes \expn^{ 
A \bar T} B -\tilde B \otimes B) \vect(I)
\end{align*}
applying (\ref{vec_kron_rel}) again, thus (\ref{kron_opt1}) follows. Condition (\ref{opt2}) is equivalent to \begin{align*}
 \vect(\tilde Q_{\bar T}\tilde B)  =  \vect(\tilde Q_{2, \bar T} B),
\end{align*}
and with property (\ref{vec_kron_rel}), it holds that \begin{align*}
  &\vect(\tilde Q_{\bar T}\tilde B)=(\tilde B^T\otimes I) \vect(\tilde Q_{\bar T})\\&=(\tilde B^T\otimes I) \left[(I\otimes D)+(D\otimes 
I)\right]^{-1}(\expn^{D \bar T}\tilde C^T \otimes \expn^{D \bar T}\tilde C^T-\tilde C^T \otimes \tilde C^T)\vect(I)
                    \end{align*}
inserting the vectorized representation of (\ref{trans_reduced_obs}). Using the identities $(\tilde B^T\otimes I)= (\hat B^T\otimes I)(S^{-T}\otimes 
I)^{-1}$ and $(\expn^{D \bar T}\tilde C^T \otimes \expn^{D \bar T}\tilde C^T -\tilde C^T \otimes \tilde C^T)=(S^{T} \otimes I)^{-1} (\expn^{\hat A^T 
\bar T}\hat C^T \otimes \expn^{D \bar T}\tilde C^T -\hat C^T \otimes \tilde C^T)$ yields
\begin{align*}
  \vect(\tilde Q_{\bar T}\tilde B)=(\hat B^T\otimes I) \left[(I\otimes D)+(\hat A^T\otimes 
I)\right]^{-1}(\expn^{\hat A^T \bar T}\hat C^T \otimes \expn^{D \bar T}\tilde C^T-\hat C^T \otimes \tilde C^T)\vect(I).
 \end{align*}
Vectorizing (\ref{trans_cross_obs}) leads to \begin{align*}
  \vect(\tilde Q_{2, \bar T}\tilde B)=(B^T\otimes I) \left[(I\otimes D)+(A^T\otimes I)\right]^{-1}(\expn^{A^T \bar T} C^T \otimes \expn^{D 
\bar T}\tilde C^T-C^T \otimes \tilde C^T)\vect(I),
                    \end{align*}
 which gives us (\ref{kron_opt2}). Condition (\ref{opt3}) is equivalent to \begin{align*}
\trace([\tilde P_{2, {\bar T}}-\bar T \expn^{A \bar T}B \tilde B^T \expn^{D \bar T}]e_i e_i^T\tilde Q_{2, \infty})=
\trace([\tilde P_{\bar T}-\bar T\expn^{D \bar T}\tilde B \tilde B^T \expn^{D \bar T}]e_i e_i^T\tilde Q_{\infty})
\end{align*}
for every $i\in \{1,\ldots,r\}$. Taking (\ref{vec_trace_rel}) into account, we can express the trace using the $\vect$ operator as follows: \begin{align}\label{eq:insertvecQ}
\trace([\tilde P_{\bar T}-\bar T\expn^{D \bar T}\tilde B \tilde B^T \expn^{D \bar T}]e_i e_i^T\tilde Q_{\infty})=
\vect^T(\tilde P_{\bar T}-\bar T\expn^{D \bar T}\tilde B \tilde B^T \expn^{D \bar T})(I\otimes e_i e_i^T)\vect(\tilde Q_{\infty}).
\end{align}
With the above arguments, we see that the vectorization of (\ref{infobsgram}) yields
 \begin{align}\label{eq:forvecQ}
 \vect(\tilde Q_{\infty})=-(S^{-T}\otimes I)  \left[(I\otimes D)+(\hat A^T\otimes I)\right]^{-1}(\hat C^T \otimes \tilde C^T)\vect(I).
                                      \end{align}
Before we proceed further, we need the following two relations:
\begin{align}\label{insertrel1}
&(S^{-1}\otimes I) \vect(\bar T\expn^{D \bar T}\tilde B \tilde B^T \expn^{D \bar T})= (\bar T\expn^{\hat A \bar T}\hat B \otimes \expn^{D \bar 
T}\tilde B) \vect(I),\\ \label{insertrel2}
& (S^{-1}\otimes I) \vect(\tilde P_{\bar T}) = \left[(I\otimes D)+(\hat A\otimes I)\right]^{-1}(\expn^{\hat A \bar T}\hat B \otimes \expn^{D \bar 
T}\tilde B -\hat B \otimes \tilde B) \vect(I).
                 \end{align}
We insert (\ref{eq:forvecQ}) into (\ref{eq:insertvecQ}) and obtain \begin{align*}
&\trace([\tilde P_{\bar T}-\bar T\expn^{D \bar T}\tilde B \tilde B^T \expn^{D \bar T}]e_i e_i^T\tilde Q_{\infty})\\&=
\vect^T(\tilde P_{\bar T}-\bar T\expn^{D \bar T}\tilde B \tilde B^T \expn^{D \bar T})(S^{-T}\otimes I)(I\otimes e_i e_i^T) 
\left[-(I\otimes D)-(\hat A^T\otimes I)\right]^{-1}\\&\quad\times(\hat C^T \otimes \tilde C^T)\vect(I).
\end{align*}
We apply (\ref{insertrel1}) and (\ref{insertrel2}) to the above identity. This leads to the following: \begin{align*}
&\trace([\tilde P_{\bar T}-\bar T\expn^{D \bar T}\tilde B \tilde B^T \expn^{D \bar T}]e_i e_i^T\tilde Q_{\infty})\\&=
\vect^T(I)\left[(\hat B^T\expn^{\hat A^T \bar T} \otimes \tilde B^T\expn^{D \bar 
T} -\hat B^T \otimes \tilde B^T) \left[(I\otimes D)+(\hat A^T\otimes I)\right]^{-1}-(\bar T \hat B^T \expn^{\hat A^T \bar T} \otimes \tilde 
B^T\expn^{D \bar 
T})\right] \\& \quad \times (I\otimes e_i e_i^T) \left[-(I\otimes 
D)-(\hat A^T\otimes I)\right]^{-1}(\hat C^T \otimes \tilde C^T)\vect(I)
\end{align*}
Using (\ref{vec_trace_rel}) and evaluating the expression \begin{align*}
\trace([\tilde P_{2, {\bar T}}-\bar T \expn^{A \bar T}B \tilde B^T \expn^{D \bar T}]e_i e_i^T\tilde Q_{2, 
\infty})=\vect^T(\tilde P^T_{2, {\bar T}}-\bar T \expn^{D \bar T}\tilde B B^T \expn^{A^T \bar T}) (I\otimes e_i e_i^T)\vect(\tilde Q_{2, \infty})
        \end{align*}
further by inserting the vectorized form of the matrices yields (\ref{kron_opt3}).
 \end{proof}
\end{thm}
\begin{remark}
The Wilson conditions (\ref{opt1}), (\ref{opt2}) and (\ref{opt3}) that are based on the finite time Gramians have been discussed in a 
talk at the SIAM Conference on Computational Science and Engineering \cite{SinaniGugercin}. Their results are indendent of this paper. 
\end{remark}
Inspired by the first-order optimality conditions as presented in Theorem \ref{thm:opt_cond} and IRKA for linear systems in \cite{morGugAB08}, we 
propose an iterative algorithm, see Algorithm \ref{algo:TL-IRKA}, which we refer to as \emph{time-limited IRKA-type algorithm}. The scheme is 
characterized by an additional term in the right side of the Sylvester equations in comparison to the classical IRKA. These Sylvester 
equations provide the projection matrices $V$ and $W$ that are used to determine the reduced system (\ref{sys:reduced}). However, we would like to 
point out that the proposed algorithm in general does not construct reduced-order systems which satisfy the first-order necessary conditions for 
optimality. Thus, our next goal is to derive expressions, which allow us to estimate how far away the obtained reduced-order systems, corresponding to 
Algorithm \ref{algo:TL-IRKA}, are from satisfying the optimality conditions exactly. 
\begin{algorithm}[!htb]
	\caption{ Time-limited IRKA-type Algorithm}
	\label{algo:TL-IRKA}
	\begin{algorithmic}[1]
		\Statex {\bf Input:} The system matrices: $ A, B,C$.
		\Statex {\bf Output:} The reduced matrices: $\hat A, \hat B,\hat C$.
		\State Make an initial guess for the reduced matrices $\hat A, \hat B,\hat C$.
\While {not converged}
		\State Perform the spectral decomposition of $\hA$ and define:
		\Statex\quad\qquad $D = S\hA S^{-1},~\tB = S\hB, ~\tC = \hC S^{-1}. $
		\State Solve for $V$ and $W$:
		
		\Statex \quad\qquad$ -V D  -  AV = B\tB^T - e^{A \bar T}B\tB^Te^{D \bar T}$,
		\Statex \quad\qquad$ -W D  -  A^T W = C^T\tC - e^{A^T\bar T}C^T\tC e^{D \bar T}$.
\State $V = \orth{(V)}$ and $W = \orth{(W)}$.
		\State Determine the reduced matrices:
		\Statex \quad\qquad $\hA = (W^T V)^{-1}W^TAV,\qquad  \hB = (W^T V)^{-1}W^TB,\qquad\hC = CV $.
		\EndWhile
	\end{algorithmic}
\end{algorithm}
\begin{thm}\label{thm:error_opt_cond}
Let $\hat A$, $\hat B$ and $\hat C$ be the reduced order matrices computed by Algorithm \ref{algo:TL-IRKA}. Then, the difference between the left and 
the right side in (\ref{kron_opt1}) is \begin{align*}
             E_c =   (I\otimes \hat C)  \left[(I\otimes \hat A)+(D\otimes I)\right]^{-1}(\expn^{D \bar T}\tilde B \otimes (W^TV)^{-1}W^T(\expn^{A 
\pro\bar T}-\expn^{A \bar T}) B ) \vect(I)                       
                                       \end{align*}
and equation (\ref{kron_opt2}) is satisfied up to the error term \begin{align*}
             E_b =  (\hat B^T\otimes I) \left[(I\otimes D)+(\hat A^T\otimes I)\right]^{-1}(V^T(\expn^{ A^T \pro^T \bar T}-\expn^{ A^T \bar T})
C^T \otimes \expn^{D \bar T}\tilde C^T)\vect(I),                        
                                       \end{align*}
where $\pro:=V(W^TV)^{-1}W^T$. For all $i=1, \ldots, r$ the deviation in (\ref{kron_opt3}) is $E_{\lambda}^i=E_{\lambda, 1}^i+E_{\lambda, 2}^i$, 
where \begin{align*}
             E_{\lambda, 1}^i=     &\vect^T(I) (\hat C \otimes \tilde C) \left[(I\otimes D)+(\hat 
A\otimes I)\right]^{-1} (I\otimes e_i e_i^T) \\ \nonumber &\times \left( \left[(I\otimes 
D)+(\hat A\otimes I)\right]^{-1}  ((W^TV)^{-1} W^T(\expn^{A \pro \bar T}-\expn^{A \bar T}) B \otimes \expn^{D \bar T}\tilde 
B)\right.\\&\quad\quad \left.-(\bar T(W^TV)^{-1} W^T(\expn^{A \pro\bar T}-\expn^{A \bar 
T}) B \otimes \expn^{D \bar T}\tilde B)\right) \vect(I)                      
              \end{align*}
and the second term is given by \begin{align*}
 E_{\lambda, 2}^i&=\vect^T(I) (C \expn^{A \bar T} \otimes \tilde C \expn^{D \bar T})\\&\quad \times
 \left[(V\otimes I) \left[(I\otimes D)+(\hat A\otimes I)\right]^{-1} ((W^T V)^{-1}W^T \otimes I)- \left[(I\otimes D)+(A\otimes 
I)\right]^{-1}\right]  \\&\quad \times (I\otimes e_i e_i^T) 
 \left[\left[(I\otimes D)+(A\otimes I)\right]^{-1} (\expn^{A \bar T} B \otimes \expn^{D \bar T}\tilde B - B \otimes \tilde B)-(\bar T\expn^{A 
\bar T} B \otimes \expn^{D \bar T}\tilde B)\right]\\&\quad \times\vect(I).                            
                               \end{align*}
\begin{proof}
The left side of (\ref{kron_opt1}) can be expressed as \begin{align*}
  (I\otimes \hat C)  \left[(I\otimes \hat A)+(D\otimes I)\right]^{-1}(\expn^{D \bar T}\tilde B \otimes (W^TV)^{-1}W^T\expn^{A \bar T} B 
-\tilde B \otimes \hat B) \vect(I)+ E_c,
                                                       \end{align*}
where we apply that $\expn^{\hat A \bar T}\hat B=(W^TV)^{-1}W^T \expn^{A \pro \bar T} B$. We set $\hat K:= (I\otimes \hat A)+(D\otimes I)$ and $K:= 
(I\otimes A)+(D\otimes I)$ and obtain \begin{align*}
  &(I\otimes \hat C)  \hat K^{-1}(\expn^{D \bar T}\tilde B \otimes (W^TV)^{-1}W^T\expn^{A \bar T} B 
-\tilde B \otimes \hat B) \vect(I)\\&= (I\otimes \hat C)  \hat K^{-1} (I\otimes (W^TV)^{-1}W^T) (\expn^{D \bar T}\tilde B \otimes \expn^{A \bar T} B 
-\tilde B \otimes B) \vect(I)\\&= (I\otimes \hat C)  \hat K^{-1} (I\otimes (W^TV)^{-1}W^T) K \vect(V)\\&=(I\otimes \hat C)  \hat K^{-1} (I\otimes 
(W^TV)^{-1}W^T) K \vect(V(W^TV)^{-1}W^T V)\\&=(I\otimes \hat C)  \hat K^{-1} (I\otimes 
(W^TV)^{-1}W^T) K (I\otimes V(W^TV)^{-1}W^T) \vect(V)\\&=(I\otimes \hat C)  \hat K^{-1}  \hat K (I\otimes (W^TV)^{-1}W^T) \vect(V)\\&=(I\otimes C) 
(I\otimes V) (I\otimes (W^TV)^{-1}W^T) \vect(V)=(I\otimes C)\vect(V)\\&= (I\otimes C) K^{-1} (\expn^{D \bar T}\tilde B 
\otimes \expn^{A \bar T} B -\tilde B \otimes B) \vect(I),
                                                       \end{align*}
where the last term above is the right side of (\ref{kron_opt1}). The left side of (\ref{kron_opt2}) is given by \begin{align*}
(\hat B^T\otimes I) \left[(I\otimes D)+(\hat A^T\otimes I)\right]^{-1}(V^T\expn^{A^T \bar T} C^T \otimes \expn^{D \bar T}\tilde C^T-\hat C^T 
\otimes \tilde C^T)\vect(I)+E_b, \end{align*}
taking the identity $\expn^{\hat A^T \bar T}\hat C^T=V^T\expn^{ A^T \pro^T \bar T} C^T$ into account. So, by setting $\hat K_2:= (I\otimes D)+(\hat 
A\otimes I)$ and $K_2:= (I\otimes D)+(A\otimes I)$, we have \begin{equation}\label{proof:Eb}\begin{aligned}
&(\hat B^T\otimes I) \hat K_2^{-T}(V^T\expn^{A^T \bar T} C^T \otimes \expn^{D \bar T}\tilde C^T-\hat C^T 
\otimes \tilde C^T)\vect(I)\\&=(\hat B^T\otimes I) \hat K_2^{-T} (V^T\otimes I) (\expn^{A^T \bar T} C^T \otimes \expn^{D \bar T}\tilde C^T-C^T 
\otimes \tilde C^T)\vect(I)\\&=(\hat B^T\otimes I) \hat K_2^{-T} (V^T\otimes I) K_2^T\vect(W^T)\\&=(\hat B^T\otimes I) \hat K_2^{-T} (V^T\otimes I) 
K_2^T\vect(W^TV(W^TV)^{-1}W^T)\\&= (\hat B^T\otimes I) \hat K_2^{-T} (V^T\otimes I) K_2^T(W(W^TV)^{-T}V^T\otimes I)\vect(W^T)\\&= (\hat B^T\otimes I) 
\hat K_2^{-T} \hat K_2^T(V^T\otimes I)\vect(W^T)\\&= (B^T\otimes I) (W(W^TV)^{-T}\otimes I) (V^T\otimes I)\vect(W^T)=(B^T\otimes I) 
\vect(W^T)\\&=(B^T\otimes I)K_2^{-T} (\expn^{A^T \bar T} C^T \otimes \expn^{D \bar T}\tilde C^T-C^T \otimes \tilde C^T)\vect(I) 
\end{aligned}\end{equation}           
 which is the right side of (\ref{kron_opt2}).  The left side of (\ref{kron_opt3}) is given by
 \begin{align*} 
   E_{\lambda, 1}^i+&\vect^T(I) (\hat C \otimes \tilde C) \hat K_2^{-1} (I\otimes e_i e_i^T) \left(\hat K_2^{-1}  
((W^TV)^{-1}W^T\expn^{A \bar T} B \otimes \expn^{D \bar T}\tilde B -\hat B \otimes \tilde 
B)\right.\\&\quad\quad\quad\quad\quad\quad\quad\quad\quad\quad\quad\quad\quad\quad\left.-(\bar T(W^TV)^{-1}W^T \expn^{A \bar T}\hat B \otimes 
\expn^{D \bar T}\tilde B)\right)\vect(I).
\end{align*}          
For the term right of $(I\otimes e_i e_i^T)$ it holds that \begin{align*} 
 &\left[\hat K_2^{-1}  ((W^TV)^{-1}W^T\expn^{A \bar T} B \otimes \expn^{D \bar T}\tilde B -\hat B \otimes \tilde 
B)\right.\\&\quad\left.-(\bar T(W^TV)^{-1}W^T \expn^{A \bar T}\hat B \otimes \expn^{D \bar T}\tilde B)\right]\vect(I)\\&= 
 \hat K_2^{-1} ((W^TV)^{-1}W^T\otimes I)  (\expn^{A \bar T} B \otimes \expn^{D \bar T}\tilde B - B \otimes \tilde 
B)\vect(I)\\&\quad-(\bar T (W^TV)^{-1}W^T \expn^{A \bar T} B \otimes \expn^{D \bar T}\tilde B)\vect(I) \\&= 
 \hat K_2^{-1} ((W^TV)^{-1}W^T\otimes I) K_2 \vect(V^T)-(\bar T (W^TV)^{-1}W^T \expn^{A \bar T} B \otimes \expn^{D \bar T}\tilde B)\vect(I)\\&= 
 \hat K_2^{-1} ((W^TV)^{-1}W^T\otimes I) K_2 \vect(V^TW(W^T V)^{-T}V^T)\\&\quad-(\bar T (W^TV)^{-1}W^T\expn^{A \bar T} B \otimes \expn^{D \bar 
T}\tilde B)\vect(I)\\&= 
 \hat K_2^{-1} ((W^TV)^{-1}W^T\otimes I) K_2 (V(W^T V)^{-1} W^T\otimes I) \vect(V^T)\\&\quad-(\bar T (W^TV)^{-1}W^T \expn^{A \bar T} B \otimes 
\expn^{D \bar T}\tilde B)\vect(I)\\&= ((W^T V)^{-1} W^T\otimes I) \vect(V^T)-(\bar T (W^TV)^{-1}W^T \expn^{A \bar T} B \otimes \expn^{D \bar T}\tilde 
B)\vect(I)\\&= (W^T V)^{-1} W^T\otimes I)\left[K_2^{-1} (\expn^{A \bar T} B \otimes \expn^{D \bar T}\tilde B - B \otimes \tilde B)-(\bar T\expn^{A 
\bar T} B \otimes \expn^{D \bar T}\tilde B)\right]\vect(I).
\end{align*}  
Since $((W^T V)^{-1} W^T\otimes I)$ and $(I\otimes e_i e_i^T)$ commute, it remains to analyze the following term \begin{align*} 
 \vect^T(I) (\hat C \otimes \tilde C) \hat K_2^{-1} ((W^T V)^{-1} W^T\otimes I)=\left[(W (W^T V)^{-T} \otimes I) \hat K_2^{-T} (\hat C^T \otimes 
\tilde C^T)\vect(I)\right]^T.
\end{align*}    
We add a zero such that \begin{align*}
  &(W (W^T V)^{-T} \otimes I) \hat K_2^{-T} (\hat C^T \otimes \tilde C^T)\vect(I)\\&=   (W (W^T V)^{-T} \otimes I) \hat K_2^{-T}(V^T\otimes I)) [(C^T 
\otimes \tilde C^T)-(\expn^{A^T \bar T} C^T \otimes \expn^{D \bar T}\tilde C^T)] \vect(I) \\ &\quad 
+(W (W^T V)^{-T} \otimes I) \hat K_2^{-T}(V^T\otimes I))  (\expn^{A^T \bar T} C^T \otimes \expn^{D \bar T}\tilde C^T) \vect(I).
                        \end{align*}
Using the same steps as in (\ref{proof:Eb}), we find \begin{align*} 
&(W (W^T V)^{-T} \otimes I) \hat K_2^{-T}(V^T\otimes I)) [(C^T 
\otimes \tilde C^T)-(\expn^{A^T \bar T} C^T \otimes \expn^{D \bar T}\tilde C^T)] \vect(I)\\&=  K_2^{-T}[(C^T \otimes \tilde C^T)-(\expn^{A^T \bar T} 
C^T \otimes \expn^{D \bar T}\tilde C^T)] \vect(I).
                                                     \end{align*}
Consequently, we have  \begin{align} \nonumber
 &\vect^T(I) (\hat C \otimes \tilde C) \hat K_2^{-1} ((W^T V)^{-1} W^T\otimes I)=\vect^T(I) (C \otimes \tilde C) K_2^{-1} \\& \label{Elambda2}
 +\vect^T(I) (C \expn^{A \bar T} \otimes \tilde C \expn^{D \bar T})\left[(V\otimes I) \hat K_2^{-1} ((W^T V)^{-1}W^T \otimes I)- K_2^{-1}\right]. 
\end{align}    
The term in (\ref{Elambda2}) provides $E_{\lambda, 2}^i$ which concludes the proof.
\end{proof}
\end{thm}
Theorem \ref{thm:error_opt_cond} allows us to point out the cases in which Algorithm \ref{algo:TL-IRKA} works well. The method is expected to perform 
well whenever the error expressions $E_b, E_c$ and $E^i_\lambda$ are small. By Theorem \ref{thm:error_opt_cond}, the error in the 
optimality condition (\ref{kron_opt1}) is bounded as follows:\begin{align*}
   \left\|E_c\right\|_2 \leq   \sqrt{m}k_c\left\|\expn^{D \bar T}\tilde B \right\|_2 \left\|(W^TV)^{-1}W^T(\expn^{A 
\pro\bar T}-\expn^{A \bar T}) B\right\|_2,                       
                                       \end{align*}
where $k_c>0$ is a suitable constant. Thus, $\left\|E_c\right\|_2$ is small if $\left\|(W^TV)^{-1}W^T(\expn^{A \pro\bar T}-\expn^{A \bar T}) 
B\right\|_2$ is small. At the same time \begin{align*}
                \left\|\expn^{D \bar T}\tilde B \right\|_2\leq \expn^{\lambda_{\max} \bar T}\left\|\tilde B \right\|_2
                \end{align*}
should not be too large which is given if the largest eigenvalue $\lambda_{\max}$ of $\hat A$ is small enough or ideally negative (asymptotic 
stability of the reduced system). Similar conclusions can be made when looking at $E_b$. It is bounded by \begin{align*}
   \left\|E_b\right\|_2 \leq   \sqrt{p}k_b\left\|\tilde C \expn^{D \bar T} \right\|_2 \left\|C(\expn^{\pro A\bar T}-\expn^{A \bar T})V\right\|_2     
                                   \end{align*}  
with a sufficiently large constant $k_b>0$. Hence, if $\left\|C(\expn^{\pro A\bar T}-\expn^{A \bar T})V\right\|_2$ is small, then condition 
(\ref{kron_opt2}) is approximately satisfied. Now, $\left\vert E_{\lambda, 1}^i\right\vert$ can be bounded in a similar way as $\left\|E_c\right\|_2$ 
such that it is also small if $\left\|(W^TV)^{-1}W^T(\expn^{A \pro\bar T}-\expn^{A \bar T}) B\right\|_2$ is neglectable, whereas for $\left\vert 
E_{\lambda, 2}^i\right\vert$ it is required to have the product \begin{align*}
&\left\|C \expn^{A \bar T}\right\|_2 \left\|\tilde C \expn^{D \bar T}\right\|_2\\&\quad \times
\left\| (V\otimes I) \left[(I\otimes D)+(\hat A\otimes I)\right]^{-1} ((W^T V)^{-1}W^T \otimes I)- \left[(I\otimes D)+(A\otimes 
I)\right]^{-1}\right\|_2                          
                               \end{align*}
small. The asymptotically stable matrix $A$ is also helpful in this context.

\definecolor{mycolor1}{rgb}{1.00000,0.00000,1.00000}%

\section{Numerical Experiments}\label{sec:numericalsection}
In this section, we investigate the efficiency of the time-limited IRKA inspired algorithm, see Algorithm~\ref{algo:TL-IRKA}, and compare it 
with conventional IRKA (unbounded time), see \cite{morGugAB08}. All the experiments are done in \matlab~8.0.0.783 (R2012b) on a machine \intel\xeon 
CPU X5650 @ 2.67GHz with 48 GB RAM. We run both iterative algorithms until the relative change in the eigenvalues of $\hA$ becomes less a 
tolerance of $10^{-8}$. We initialize conventional IRKA randomly, and we use the reduced-order system obtained by conventional 
IRKA as an initial guess for Algorithm~\ref{algo:TL-IRKA}.  In Table \ref{tab:list_examples}, we list the examples used in order to compare the 
algorithms. For all examples, we compare the impulse responses of the systems, which is simulated using the \texttt{impulse} command from MATLAB. To 
quantify the quality of reduced-order systems, we determine either the absolute or the relative error, depending on weather the impulse response 
crosses zero or not. We define the absolute $\cE^{(a)}(t)$ and relative errors $\cE^{(r)}(t)$, respectively, as follows: 
\begin{equation}
\cE^{(a)}(t) := \|y^{(\delta)}(t) - y^{(\delta)}_r(t)\| \quad \text{and} \quad \cE^{(r)}(t) := \dfrac{\|y^{(\delta)}(t) - y^{(\delta)}_r(t)\|}{\|y(t)\|},
\end{equation}
where $y^{(\delta)}$ and $y^{(\delta)}_r$ are the impulses responses of original and reduced-order systems.  In addition to this, we numerically 
examine how far away the reduced-order systems due to IRKA and Algorithm~\ref{algo:TL-IRKA} are from satisfying the optimality conditions 
\eqref{kron_opt1} -- \eqref{kron_opt3}. To measure this, we first define the following quantities:
\begin{subequations}\label{eq:def_error}
\begin{align}
\mathcal E_c &= \|\mathcal R^{(c)}_{l}- \mathcal R^{(c)}_{r} \| /  \|\mathcal R^{(c)}_{l} \|,  \\
\mathcal E_b &= \|\mathcal R^{(b)}_{l}- \mathcal R^{(b)}_{r} \| /  \|\mathcal R^{(b)}_{l} \|, \\
\mathcal E_\lambda &= \max_{i}{(\mathcal R_{\lambda_{i}}) },  &\mathcal R_{\lambda_{i}} &= \left\vert\mathcal R^{(\lambda_i)}_{l}- \mathcal 
R^{(\lambda_i)}_{r} \right\vert /  \left\vert\mathcal  R^{(\lambda_i)}_{l} \right\vert,
\end{align}
\end{subequations}
where $\mathcal R^{(c)}_{l}$ and $\mathcal R^{(c)}_{r}$ are the left and right sides of \eqref{kron_opt1}; $\mathcal R^{(b)}_{l}$ and $\mathcal 
R^{(b)}_{r}$ are the left and right sides of \eqref{kron_opt2}; $\mathcal R^{(\lambda_i)}_{l}$ and $\mathcal R^{(\lambda_i)}_{r}$ are the left and 
right sides of \eqref{kron_opt3};  $\max(\cdot)$ denotes the maximum. 
\begin{table}[tb!]
	\centering
	\begin{tabular}{|c|c|c|c|}
\hline
Example & n  & m & p\\
\hline
Heat equation & 200 & 1& 1\\
\hline 
Clamped beam model & 348 & 1& 1\\
\hline
Component $1$r of the International Space Station & 270 & 3&3\\
\hline
	\end{tabular}
\caption{A list of examples with their dimensions $(n)$, the number of inputs $(m)$ and outputs $(p)$. These examples are taken from 
\url{http://slicot.org/20-site/126-benchmark-examples-for-model-reduction}.}
\label{tab:list_examples}
\end{table}

In the following, we discuss each of these examples in detail. Beginning with the heat example, we compute the reduced-order systems by employing 
conventional IRKA and Algorithm~\ref{algo:TL-IRKA} of order $r = 5$. We consider the terminal time $\bar T = 1$. In Figure \ref{fig:heat_impulse}, 
we compare the impulse response which shows that Algorithm~\ref{algo:TL-IRKA} yields a reduced-order system, replicating the systems dynamics better 
in the time interval $[0,\bar T]$. Furthermore, as it has been noted in Section \ref{sec:optimality_condtions}, Algorithm~\ref{algo:TL-IRKA} does not 
yield a reduced-order system, satisfying the optimality conditions. Thus, in Table \ref{tab:heat_opt} we measure the error of the 
reduced-order systems obtained via IRKA and Algorithm~\ref{algo:TL-IRKA} in the optimality conditions as described in \eqref{eq:def_error}. The table 
shows that for the heat example, Algorithm~\ref{algo:TL-IRKA} does a better job in satisfying the two optimality conditions, and in contrast the 
third condition is satisfied better by the reduced-order system due to conventional IRKA.
\begin{figure}[!htb]
	\centering
	\includegraphics{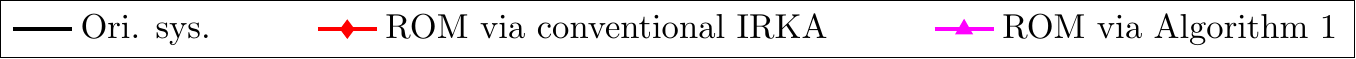}
	\centering
	\setlength\fheight{3cm}  \setlength\fwidth{5.25cm}
	\includegraphics{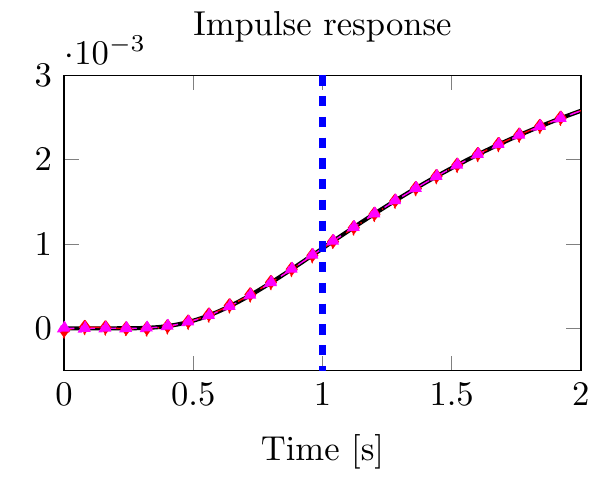}  \includegraphics{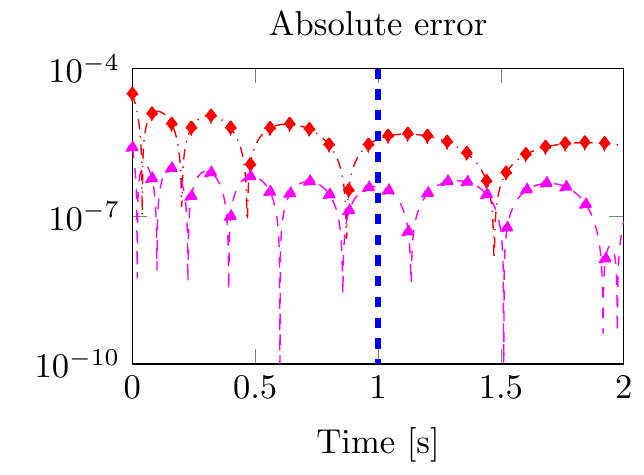}
	\caption{Heat example: a comparison of the impulse response of the original system and reduced-order system obtained via IRKA and 
Algorithm~\ref{algo:TL-IRKA}.}
	\label{fig:heat_impulse}
\end{figure}

	\begin{table}[!tb]
		\centering
		\begin{tabular}{|c|c|c|c|}
			\hline
			Method & $\cE_c$ & $\cE_b$ & $\cE_\lambda$\\ 
			\hline
IRKA & $2.7\times 10^{-3}$ & $2.7\times 10^{-3}$ & $9.10 \times 10^{-3}$\\ 
\hline
TL-IRKA &  $1.39\times 10^{-4}$ & $1.39\times 10^{-4}$ & $1.58\times 10^{-1}$\\
\hline
		\end{tabular}
		\caption{Heat example: relative errors in satisfying the optimality conditions.}
		\label{tab:heat_opt}
	\end{table}

As a second example, we have taken a beam model which is reduced to the order $r = 10$ using the IRKA and Algorithm~\ref{algo:TL-IRKA}. For this, we 
set the terminal time to $\bar T = 2$. Next, we compare the impulse responses of the original and reduced-order systems in Figure 
\ref{fig:beam_impulse}. Clearly, we observe that Algorithm~\ref{algo:TL-IRKA} produces a better reduced-order system as compared to IRKA at least 
within the time interval of interest. Furthermore, in Table \ref{tab:beam_opt}, we measure the error of the obtained reduced-order systems 
in the optimality conditions, where we make a similar observation as in the heat example. 

\begin{figure}[!htb]
	\centering
	\includegraphics{legend1.pdf}
	\centering
	\setlength\fheight{3cm}  \setlength\fwidth{5.25cm}
	\includegraphics{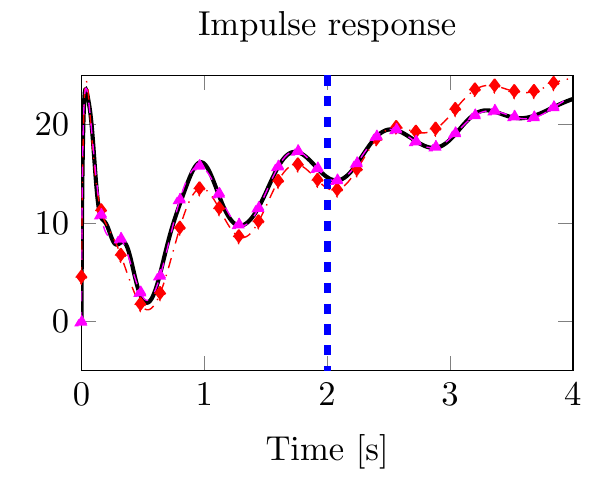}  \includegraphics{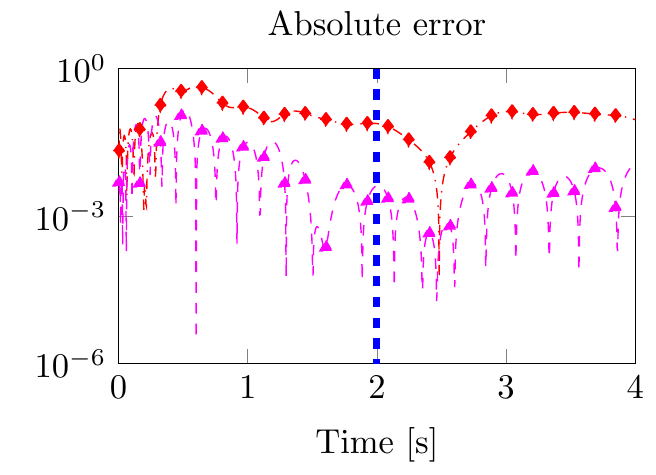}
	\caption{Beam example: a comparison of the impulse response of the original system and reduced-order system obtained via IRKA and 
Algorithm ~\ref{algo:TL-IRKA}.}
	\label{fig:beam_impulse}
\end{figure}

	\begin{table}[!bt]
		\centering
		\begin{tabular}{|c|c|c|c|}
			\hline
			Method & $\cE_c$ & $\cE_b$ & $\cE_\lambda$\\ \hline
			IRKA & $5.96\times 10^{-2}$ & $5.96\times 10^{-2}$ & $9.47 \times 10^{-2}$\\ 
			\hline
			TL-IRKA &  $3.94\times 10^{-4}$ & $3.94\times 10^{-4}$ &$1.26\times 10^{-1}$\\
			\hline
		\end{tabular}
			\caption{Beam example: relative error in satisfying the optimality conditions.}
			\label{tab:beam_opt}
	\end{table}

Lastly, we present the results for the model of a space station.  We first set the terminal time to $\bar T = 1$.  For this example, we construct 
reduced systems of order $r = 20$ via IRKA and Algorithm~\ref{algo:TL-IRKA} and compare the quality of them using the impulse response. Since the 
example has $3$ inputs and $3$ outputs, for brevity we refrain to plot the impulse response, but we rather plot the norm absolute error which is shown 
in Figure \ref{fig:ISS_imp_1s}. We observe that Algorithm~\ref{algo:TL-IRKA} constructs a reduced-order system which replicates the dynamics better 
within the time interval of interest.  For this example, we again compute how far away the reduced-order systems are from satisfying the optimality 
conditions exactly in Table \ref{tab:ISS_opt_1s}. For this example as well, Algorithm \ref{algo:TL-IRKA} does a better job than IRKA in satisfying 
the first two conditions, but fails to perform better for the third conditions. However, importantly, Algorithm \ref{algo:TL-IRKA} yields a better 
reduced-order system. 


\begin{figure}[!tb]
	\centering
\includegraphics{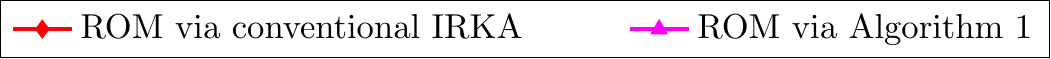}
\centering
	\setlength\fheight{3cm}  \setlength\fwidth{5.25cm}
	\includegraphics{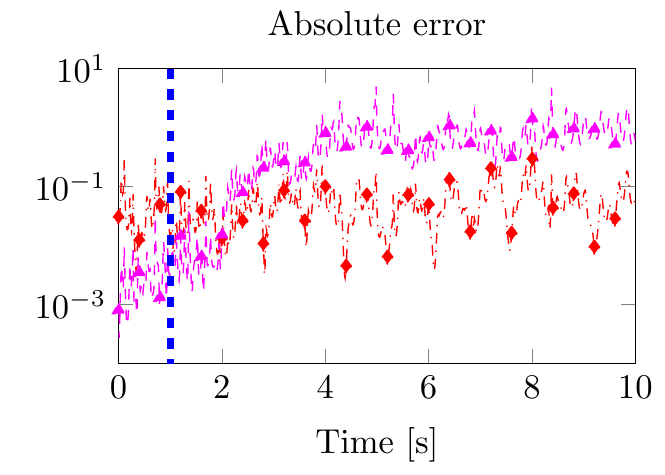}
	\caption{ISS example: a comparison of the impulse response of the original system and reduced-order system obtained via IRKA and 
Algorithm~\ref{algo:TL-IRKA}.}
	\label{fig:ISS_imp_1s}
\end{figure}

	\begin{table}[H]
	\centering
	\begin{tabular}{|c|c|c|c|}
		\hline
		Method & $\cE_c$ & $\cE_b$ & $\cE_\lambda$\\ \hline
		IRKA & $2.61\times 10^{-1}$ & $1.62\times 10^{-1}$ & $1.08 \times 10^{-1}$\\ 
		\hline
		TL-IRKA &  $6.00\times 10^{-2}$ & $5.43\times 10^{-3}$ & $ 4.46\times 10^{-1}$\\
		\hline
	\end{tabular}
	\caption{ISS example: relative error in satisfying the optimality conditions.}
\label{tab:ISS_opt_1s}
\end{table}

%

\section{Conclusions}

In this work, we have studied large scale linear time-invariant systems which we aimed to reduce. We showed that the error between the original 
and the reduced system on a finite time interval can be bounded using the so-called time-limited $\mathcal H_2$-norm. In order to find a reduced 
order model with a small output error, we minimized the $\mathcal H_2$-norm with respect to the reduced order system matrices. This resulted in 
necessary conditions for optimality using representation of the time-limited $\mathcal H_2$-norm based on the time-limited Gramians. Reduced systems 
satisfying theses condition are expected to perform well on the finite time interval of interest. Based on these optimality conditions, we propose an 
iterative scheme which is inspired by the iterative rational Krylov algorithm \cite{morGugAB08}. Moreover, the error of the proposed iterative 
algorithm in the derived optimality conditions has been analyzed to point out the cases in which the proposed method works 
particularly well. We concluded this paper by comparing conventional IRKA, an algorithm leading to a good reduced system on an infinite time 
horizon, with the proposed iterative scheme in several numerical experiments. The simulations showed that time-limited IRKA can outperform IRKA on the 
finite time interval of interest.

As we have seen, the proposed iterative-type algorithm for the time-limited problem does not satisfy the optimality conditions exactly. Therefore, it 
would be worthwhile to come up with an improved algorithm, allowing us to construct a reduced-order system which satisfy the derived optimality 
conditions exactly. 

\bibliographystyle{plain}

\end{document}